\documentclass{article}

\usepackage{amsmath, amssymb, amsthm,tikz}
\usepackage{mathtools}
\usepackage{comment}
\usepackage[all,cmtip]{xy}
\usepackage{xcolor}
\usepackage[utf8]{inputenc}
\usepackage{tikz}
\usepackage{tikz-cd}
\usetikzlibrary{matrix, calc, arrows}
\usetikzlibrary{arrows}
\usepackage{pgfplots}
\usepgfplotslibrary{patchplots}
\usetikzlibrary{patterns, positioning, arrows}
\pgfplotsset{compat=1.15}
\usepackage{hyperref}
\newcommand{\curv}{1.5pc}

\title{On smooth and bundle structures on \\ topological manifolds homotopy equivalent \\ to $S^{4k-1}$-bundles over $S^{4k}$}
\author{Ajay Raj, Tibor Macko}

\date{\today}

\newtheorem{theorem}{Theorem}[section]
\newtheorem{lemma}[theorem]{Lemma}
\newtheorem{proposition}[theorem]{Proposition}
\newtheorem{corollary}[theorem]{Corollary}

\theoremstyle{definition}

\newtheorem{remark}[theorem]{Remark}

\newcommand{\R}{\mathbb{R}}

\newcommand{\Z}{\mathbb{Z}}

\newcommand{\BCAT}{\textup{BCAT}}
\newcommand{\BTOP}{\textup{BTOP}}
\newcommand{\BG}{\textup{BG}}
\newcommand{\CAT}{\textup{CAT}}
\newcommand{\TOP}{\textup{TOP}}
\newcommand{\DIFF}{\textup{DIFF}}
\newcommand{\G}{\textup{G}}
\newcommand{\GmodTOP}{\G/\TOP}
\newcommand{\BPL}{\textup{BPL}}
\newcommand{\PL}{\textup{PL}}
\newcommand{\BO}{\textup{BO}}
\newcommand{\BSO}{\textup{BSO}}
\newcommand{\Orthogroup}{\textup{O}}
\newcommand{\GmodO}{\G/\Orthogroup}
\newcommand{\TOPmodO}{\TOP/\Orthogroup}
\newcommand{\SO}{\textup{SO}}
\newcommand{\SF}{\textup{SF}}
\newcommand{\SH}{\textup{SH}}

\newcommand{\im}{\textup{im}}
\newcommand{\id}{\textup{id}}

\newcommand{\sS}{\mathcal{S}}

\begin{document}
	
	\maketitle
	
	\begin{abstract}
			We first verify the expected calculations of the topological structure set, in the sense of surgery theory, for the total spaces of $S^{4k-1}$-bundles over $S^{4k}$ for $k\ge3$, in which manifolds homotopy equivalent to the total spaces are organized. Next, we investigate the question of which of the elements in these structure sets can be realized as such bundles. Moreover, we study the forgetful map from the smooth version to the topological version of the structure set; we determine its image. All elements in the image turn out to have the underlying manifold such a bundle.
	\end{abstract}
	\let\thefootnote\relax\footnotetext{This work was supported by the grant VEGA 1/0425/25.}
	\let\thefootnote\relax\footnotetext{Keywords and phrase: vector bundle, sphere bundle over sphere, microbundle, homotopy equivalence, homeomorphism, surgery, characteristic class}
	\let\thefootnote\relax
	\newcommand{\Addresses}{
		\bigskip
		\footnotesize
		
		Ajay Raj, 
		
		\textsc{Department of Applied Sciences and Humanities, PIT, Parul University,\\ Vadodara, Gujarat, India-391760}\par\nopagebreak
		\textit{E-mail address}: \texttt{ajayraj4august@gmail.com}
		
		\medskip
		
		Tibor Macko, 
		
		\textsc{Department of Algebra and Geometry, FMFI, Comenius University, Bratislava, SK-84248, Slovakia, and}
		
		\textsc{Institute of Mathematics, Slovak Academy of Sciences, \v Stef\'anikova 49, Bratislava, SK-81473, Slovakia}\par\nopagebreak
		\textit{E-mail address}: \texttt{tibor.macko@fmph.uniba.sk}
		
		
	}
	\section{Introduction}	
	
	In \cite{RM24} we calculated the topological surgery structure sets of the total spaces of $S^7$-bundles over $S^8$ with structure group $\SO(8)$, and we determined which of their elements can be realized as such bundles. Moreover, we studied the forgetful map from the smooth version to the topological version of the structure set, and showed that it is surjective in some special cases (see Remark 3.5 in \cite{RM24}). The purpose of this article is to generalize that work to $S^{4k-1}$-bundles over $S^{4k}$ with structure group $\SO(4k)$ for $k \ge 3$.
	
	Let $X$ be a closed topological manifold of dimension $d$. The topological structure set of $X$, denoted by $\mathcal{S}^{\TOP}(X)$, is the collection of equivalence classes $[(Y,f)]$ of homotopy equivalences $f: Y \rightarrow X$ with $Y$ a closed topological manifold of dimension $d$, where $(Y_1,f_1)$ and $(Y_2,f_2)$ are equivalent if there exists a homeomorphism $\phi:Y_1 \rightarrow Y_2$ such that $f_{2} \circ \phi \simeq f_1$. Given $X$, a closed smooth manifold of dimension $d$, the smooth structure set of $X$, denoted by $\mathcal{S}^{\DIFF}(X)$, is defined analogously using smooth manifolds and diffeomorphisms. Understanding these structure sets brings us close to the full classification of manifolds in a given homotopy type in the respective categories, and an understanding of the forgetful map $F_X \colon \mathcal{S}^{\DIFF}(X) \rightarrow \mathcal{S}^{\TOP}(X)$ yields an answer to the question of the existence and uniqueness of a smooth structure on topological manifolds in our homotopy type. Surgery theory is a tool designed to investigate these questions. It was applied in the past for many specific examples of such $X$; see, e.g., \cite[Chapter 18]{LM24} or \cite{CW2021} for a recent survey. However, there are also interesting examples that have not yet appeared in the literature. Calculations are harder in the smooth case than in the topological case and, in fact, they are much scarcer; see \cite[Section 18.12]{LM24}. 
	
	In this paper, we focus on the case when $X$ is the total space of an orthogonal $S^{4k-1}$-bundle over $S^{4k}$ for $k \geq 3$. On one hand, these spaces form a well known class of manifolds that have been extensively studied in algebraic topology since its ancient times; see \cite{JW54,JW55}. On the other hand, a detailed study from the point of view of surgery was missing. Our main results, Theorems \ref{thm:top-str-set} and \ref{thm:the-forgetful-map-on-str-sets} below, determine the topological structure set and the image of the forgetful map. Although we do not fully determine the smooth structure sets, we provide important information that will be needed in their study, which we hope to work on in the future. 
	
	Isomorphism classes of oriented $S^{4k-1}$-bundles over $S^{4k}$ with the structure group $\SO(4k)$ are in one-to-one correspondence with the homotopy group $\pi_{4k-1}(\SO(4k))$. Suppose $k \geq 3$. We shall use the isomorphism: 
	\begin{equation}\label{z2z}
		(S,e):\pi_{4k-1}(\SO(4k)) \xrightarrow{\cong} \Z \oplus 2\Z
	\end{equation}
	given as follows. The first factor is $S \colon \pi_{4k-1}(\SO(4k)) \rightarrow \pi_{4k-1}(\SO) \rightarrow \Z$, with the first map induced by the canonical inclusion and the second one being the well-known isomorphism from Bott periodicity. (The letter $S$ stands for ``stabilization''.) The factor $e$ denotes the Euler number, which is known to be always even when $k \geq 3$; hence, it takes values in $2\Z$. The fact that this is an isomorphism can be obtained, for example, from the commutative diagram $(\ast)$ on page 182 in \cite{CCPS23} and we also discuss this in more detail in Section~\ref{subsec:sphere-bdles-and-spherical-fibrations}. 
	
	In other words, we have chosen two generators $\sigma$ and $\rho$, where, say $\sigma$, is sent to $(0,2)$, and the other, $\rho$, is sent to $(1,0)$. (Actually, $\sigma$ turns out to be the tangent bundle $TS^{4k}$; see Section~\ref{subsec:sphere-bdles-and-spherical-fibrations}.) Using this isomorphism, for each pair of integers $(m,n)$, we get a vector bundle 
	\[
	m\rho + n\sigma = \gamma_{m,n}^k:= (\R^{4k}\hookrightarrow E_{m,n}^k \xrightarrow{p_{m,n}} S^{4k}) \in \pi_{4k-1}(\SO(4k)).
	\]
	From this vector bundle, we obtain the corresponding sphere bundle 
	\[
	S(\gamma_{m,n}^k):= S^{4k-1} \hookrightarrow M_{m,n}^k \xrightarrow{\overline{p}_{m,n}} S^{4k}
	\]
	and the corresponding disk bundle $D(\gamma_{m,n}^k):= D^{4k} \hookrightarrow W_{m,n}^k \xrightarrow{\pi_{m,n}} S^{4k}$. The Euler number for the bundle is $e(\gamma_{m,n}^k)=2n$. 
	Note that a change in the orientation of the fiber yields the homeomorphism $M_{m,n}^k \cong M_{m+n,-n}^k$, and a change in the orientation of the base yields $M_{m,n}^k \cong M_{-m,-n}^k$. Therefore, we may assume $n\ge0$. 
	
	We proceed to our main results (using the above notation and keeping the assumption $k \geq 3$). For the first one, let us recall that the topological version of the structure set $\mathcal{S}^{\TOP} (X)$ always possesses a natural abelian group structure; see \cite[Section 11.8]{LM24}.  
	
	\begin{theorem}\label{thm:top-str-set}
		There is an isomorphism $\mathcal{S}^{\TOP}(M_{m,n}^k) \cong \Z/e\Z$ where $e = 2n$ is the Euler number of $\gamma_{m,n}^k$.
	\end{theorem}
	
	In general,  there is no natural group structure on the smooth version of the structure set $\mathcal{S}^{\DIFF} (X)$. Hence the forgetful map $F_X \colon \mathcal{S}^{\DIFF} (X) \rightarrow \mathcal{S}^{\TOP} (X)$ cannot be a homomorphism. However, for $X = M_{m,n}^k$ we show in Corollary~\ref{cor:image-of-F-is-subgroup} that the image of $F_{M_{m,n}^k}$ is a subgroup and hence we can talk about the ``cokernel'' of this map. 
	
	\begin{theorem}\label{thm:the-forgetful-map-on-str-sets}
		The forgetful map $F_{M_{m,n}^k} \colon \mathcal{S}^{\DIFF}(M_{m,n}^k) \rightarrow \mathcal{S}^{\TOP}(M_{m,n}^k)$ has its cokernel isomorphic to $\Z/g\Z$ where $g$ is given as follows. 
		
		When $n \neq 0$, then $g$= gcd$(e,\theta_k)$ with $e=2n$ the Euler number of the bundle and 
		\[
		\theta_k = \text{numerator}\left(\frac{B_k}{4k}\right)\cdot a_k \cdot 2^{2k-2}\cdot(2^{2k-1}-1)
		\]
		where $B_k$ is $k$-th Bernoulli number and $a_k$ is $1$ or $2$ according to $k$ is even or odd.
		
		When $n=0$, then $g=\theta_k$. 
	\end{theorem}
	
	The occurrence of the Bernoulli numbers and the numbers $\theta_k$ is related to their appearance in the work of Kervaire-Milnor on exotic spheres, which will be recalled in more detail in Subsection~\ref{subsec:sKM-braid-and-Brumfiel-results}. 
	
	These theorems turn out to have corollaries that are interesting in their own right. Here are some of them: 
	
	\begin{corollary}\label{cor:never-surjective}
		The forgetful map $F_{M_{m,n}^k}$ is never surjective.
	\end{corollary}
	
	\begin{corollary}\label{cor:any-smooth-is-bundle}
		Any element in the image of $F_{M_{m,n}^k}$ can be realized with the source manifold as an $S^{4k-1}$-bundle over $S^{4k}$ with the structure group $\SO(4k)$.
	\end{corollary}
	
	\begin{corollary}\label{cor:if-n-theta-coprime-then-half-of-elements-are-bundles}
		If gcd$(n,\theta_k)=1$ then half of the elements of $\mathcal{S}^{\TOP}(M_{m,n}^k)$ can be realized with the source manifold as an $S^{4k-1}$-bundle over $S^{4k}$.
	\end{corollary}
	
	An additional, more technical result is discussed in Section~\ref{sec:leftovers}.
	
	We would also like to mention further related recent work. The article \cite{biswas2026} treats the case $(m,n)=(0,0)$, which yields $M_{(0,0)} = S^{4k-1} \times S^{4k}$. The paper \cite{zhu-pan(2026)} provides homotopy classification of $S^{2n-1}$-bundles over $S^{2n}$. Combining that work in the case when $n=2k$ with our work could help in the homeomorhism or diffeomorphism classification of such bundles.
	
	The paper is structured as follows. In Section \ref{sec:tools}, we recall the necessary tools, namely the basic surgery setup, basic results about vector bundles, sphere bundles, and spherical fibrations and their relationship, which involves the well-known $J$-homomorphism from homotopy theory, and results of Brumfiel about the Kervaire-Milnor braid from \cite{GB68}. Theorem \ref{thm:top-str-set} is proved here using the standard surgery theory setup. In Section \ref{sec:realisation_results} we generalize a method by Crowley and Escher \cite{CE03} to produce specific elements in the topological structure set $\mathcal{S}^{\TOP}(M_{m,n}^k)$ with the underlying manifold $M_{m',n}^k$ for some other $m'$. The key idea is to view sphere bundles as spherical fibrations. When we do this, different sphere bundles may produce fiber homotopy equivalent spherical fibrations, whose total spaces are obviously homotopy equivalent. What we do is identify homotopy equivalences obtained in this way as certain elements in the structure set. All these examples will be smooth by construction, giving us a lower bound for the image of the forgetful map; the final result is in Theorem~\ref{thm:alpha}. In Section \ref{sec:forgetful_map}, we prove Theorem~\ref{thm:the-forgetful-map-on-str-sets}. Surgery theory allows us to translate the desired statement to a statement about a map between the groups of homotopy classes of maps from $M_{m,n}^{k}$ into certain classifying spaces $\G/\CAT$, where $\G/\CAT = \GmodO$ or $\G/\CAT = \GmodTOP$ are well known, see Theorem~\ref{fn}. This map is then studied using various exact sequences that we obtain from a certain stable homotopy decomposition of $M_{m,n}^k$ and from relating the spaces $\G/\CAT$ to other classifying spaces appearing in the Kervaire-Milnor braid and utilizing the results of Brumfiel. Proof of Theorem~\ref{thm:the-forgetful-map-on-str-sets} as well as proofs of Corollaries~\ref{cor:never-surjective}, \ref{cor:any-smooth-is-bundle}, \ref{cor:if-n-theta-coprime-then-half-of-elements-are-bundles} using Theorem~\ref{fn} are at the end of the section. Section~\ref{sec:leftovers} contains a discussion of some additional topics. 
	
	\subsection{Acknowledgments}
	
	The present paper is based on the PhD thesis of AR defended in the summer of 2025, TM was the advisor of AR. The authors would like to thank Diarmuid Crowley for useful suggestions regarding the material in Section~\ref{sec:forgetful_map}.

	
	\section{Tools}\label{sec:tools}
	

	In this section, we shall establish the tools that will be used in proving our results. Let us start with the cohomology of the total spaces $M_{m,n}^k$.
	
	Let $n >0$ and assume $k \geq 3$.  We already know that $e(\gamma_{m,n}^k) = 2n$. Now standard algebraic topology yields the following isomorphisms:
	\begin{align*}
		H^{0}(M_{m,n}^k) & \cong H^{8k-1}(M_{m,n}^k) \cong \Z; \\
		H^{4k}(M_{m,n}^k) & \cong  \Z/e\Z; \\
		H^{i}(M_{m,n}^k) & \cong 0 \text{ for } i \neq 0, 4k, 8k-1,
	\end{align*}
	where the second isomorphism is established using the Gysin sequence. 
	
	When $n=0$, we get $H^{j}(M_{m,0}^k)\cong \Z$ when $j=0,4k-1,4k$ and $8k-1$ and $H^{j}(M_{m,0}^k)\cong 0$ otherwise. 
	
	In the sequel, we will use the symbol $\omega_{4k}$ for the standard generator of $H^{4k}(S^{4k};\Z)$. 
	
	\subsection{Simply connected surgery theory}\label{subsec:simply-ctd-surgery-theory}
	
	Next, let us recall some surgery theory. Let $\CAT = \DIFF$ or $\TOP$. Given a closed $d$-dimensional $\CAT$-manifold $X$, there is the $\CAT$-surgery exact sequence \cite[Chapter 11]{LM24}
	\begin{equation} \label{eqn:cat-ses}
		\cdots \xrightarrow{\sigma_{\ldots}} L_{d+1}(\Z[\pi_{1}(X)])\xrightarrow{\partial_{d+1}} \mathcal{S}^{\CAT}(X)\xrightarrow{\eta_d} \mathcal{N}^{\CAT}(X)\xrightarrow{\sigma_d} L_{d}(\Z[\pi_{1}(X)]).
	\end{equation}
	When $\CAT=\TOP$ this is an exact sequence of abelian groups, when $\CAT=\DIFF$ this is only an exact sequence of pointed sets for general $X$, see \cite[Sec. 11.8]{LM24}. The manifolds we study are $X=M_{m,n}^k$ and their dimension is $d=8k-1$. The long exact sequence of homotopy groups of a fiber bundle implies that they are simply-connected. Therefore, the relevant $L$-groups are 
	\[
	L_d (\Z) \cong \Z,0,\Z/2,0 \quad \textup{for} \quad d \equiv 0,1,2,3 \mod 4.
	\]
	When $\CAT=\TOP$, the simply-connected surgery obstruction maps $\sigma_d$ are surjective for all $d$. Hence, we obtain a well known fact (see \cite[Theorem 1.65]{CW2021}) that for odd-dimensional, simply-connected (topological) $d$-manifolds the map $\eta_d$ is an isomorphism:
	\begin{equation}
		\eta_d \colon \mathcal{S}^{\TOP}(X) \xrightarrow{\cong} \mathcal{N}^{\TOP}(X).
	\end{equation}
	The set (group) $\mathcal{N}^{\CAT}(X)$, known as the set of normal invariants, is defined as a certain set of bordism classes of degree one normal maps in the sense of surgery. However, for us, its homotopy theoretic description via the standard Pontryagin-Thom type construction as the set of homotopy classes of maps will be of higher importance:
	\begin{equation} \label{eqn:ni-is-maps-to-G-mod-CAT}
		\mathcal{N}^{\CAT}(X) \cong [X,\G/\CAT].
	\end{equation}
	Here $\G/\CAT$ is the homotopy fiber of the canonical map $j: \BCAT \rightarrow \BG$. If $\CAT = \TOP$, then $\BTOP$ is the classifying space for stable topological microbundles (the right type of bundles for topological manifolds; see \cite{KS77}) and $\BG$ is the classifying space of stable spherical fibrations. If $\CAT = \DIFF$, then $\BCAT = \BO$ is the classifying space for stable vector bundles. 
	
	The space $\GmodO$ is an infinite loop space, where the $H$-space structure is obtained from the Whitney sum operations on $\BO$ and $\BG$. The space $\GmodTOP$ turns out to have two different infinite loop space structures. The $H$-spaces structure of one of them is also obtained from the Whitney sum operations on $\BTOP$ and $\BG$. The second one, which is more subtle, comes from topological surgery theory \cite[Sec. 11.8]{LM24}. These $H$-space structures induce a-priori different group structures on $[X,\GmodTOP]$. However, it turns out that for $X = S^l$ and for $X = M_{m,n}^{k}$ these group structures agree, see Lemma~\ref{lem:group-structures-on-ni-of-E}.
	
	If our $d$-dimensional compact manifold $X$ has a boundary and $i:\partial X \hookrightarrow X$ is the inclusion of the boundary, then there exist several versions of the structure sets for $X$ and the corresponding surgery exact sequences. Here we need the one, which we denote by $\mathcal{S}^{\TOP}(X)$, where elements are represented by homotopy equivalences $(f,\partial f) \colon (Y,\partial Y) \rightarrow (X,\partial X)$ with $(Y,\partial Y)$ another compact manifold with boundary. Note that $\partial f$ is not assumed to be a homeomorphism (there is a version of the structure set with such a condition, see \cite[Ch. 11]{LM24}, but we will not use it). There is the corresponding notion of normal invariants and $L$-groups and the surgery exact sequence, see \cite{Wall99}, \cite[Ch. 11.11]{LM24}. We have an obvious  commutative diagram:
	\begin{equation}\label{bdry}
		\begin{tikzcd}[row sep=1em,column sep=1em]
			\mathcal{S}^{\TOP}(X) \arrow[d,"i^*"] \arrow[r,"\eta"] & \mathcal{N}^{\TOP}(X) \arrow[d,"i^*"] \\
			\mathcal{S}^{\TOP}(\partial X) \arrow[r,"\eta"] & \mathcal{N}^{\TOP}(\partial X) 
		\end{tikzcd}
	\end{equation}
	
	\noindent \textbf{Proof of Theorem 1.1:} The proof proceeds by applying the surgery theory described above to the pair $(W_{m,n}^k,M_{m,n}^k)$ of the disk bundle and sphere bundle defined earlier. Note that $\partial W_{m,n}^k=M_{m,n}^k$. The $\pi-\pi$ Theorem (see \cite[Theorem 1.65]{CW2021}) implies that in this case the top horizontal map $\eta$ in~\eqref{bdry} is an isomorphism. Furthermore, primary obstruction to null homotopy implies (see \cite[Theorem 13.11, Chapter VII]{GEB}):
	\begin{equation}\label{idfic}
		[W_{m,n}^{k},\GmodTOP]\cong H^{4k}(W_{m,n}^k;\Z) \text{ and } [M_{m,n}^k,\GmodTOP] \cong H^{4k}(M_{m,n}^k;\Z) 
	\end{equation}
	Also note that for any $d>0$, $\pi_{d}(\GmodTOP) \cong L_{d}(\Z)$. Therefore for $n>0$, we get the following commutative diagram:
	
	\begin{equation}\label{equation:ses}
		\begin{tikzcd}[row sep=1em,column sep=1em]
			\mathcal{S}^{\TOP}(W_{m,n}^k) \arrow[d,"i^*"] \arrow[r,"\overset{\eta}{\cong}"] & \mathcal{N}^{\TOP}(W_{m,n}^k) \arrow[d,"i^*"] \arrow[r,"\cong"] & H^{4k}(W_{m,n}^k) \arrow[d,"i^*"] \arrow[r,"\cong"] & \Z \dar  \\
			\mathcal{S}^{\TOP}(M_{m,n}^k) \arrow[r,"\overset{\eta}{\cong}"] & \mathcal{N}^{\TOP}(M_{m,n}^k) \arrow[r,"\cong"] & H^{4k}(M_{m,n}^k) \arrow[r,"\cong"] & \Z/e\Z,
		\end{tikzcd}
	\end{equation}
	The right vertical map is the canonical projection. When $n=0$, we shall have $\Z$ in the bottom right corner, and the right vertical map will be the identity. The maps $i^*$ are induced by restriction to the boundary. The desired statement follows from the bottom line of Diagram \eqref{equation:ses}.  \qed
	
	\begin{remark}[$\DIFF$-version]
		The  smooth version $\sS^{\DIFF} (X)$ of the structure set is, in general, harder to determine. Firstly, the $\DIFF$-version of the surgery exact sequence~\eqref{eqn:cat-ses} is a-priori not an exact sequence of abelian groups, only of pointed sets. Secondly, the set of normal invariants $\mathcal{N}^{\DIFF}(X) \cong [X,\GmodO]$ is harder to determine, since the homotopy type of $\GmodO$ is more complicated than that of $\GmodTOP$, see~\cite[Ch. 5]{MM}. Finally, although the $L$-groups are the same, the map $\eta$ need not be injective if $X$ is odd-dimensional. For our $X = M^k_{m,n}$ we will be able to overcome some of these difficulties and obtain some information about $\sS^{\DIFF} (M^k_{m,n})$, but the full calculation will remain open.
	\end{remark}
	
	\subsection{Orthogonal sphere bundles and spherical fibrations}\label{subsec:sphere-bdles-and-spherical-fibrations}
	
	Our next goal is to study the elements of $\sS^{\TOP}(M_{m,n}^k)$ in more detail. We will construct such elements using the idea that two distinct sphere bundles may become fiber homotopy equivalent as spherical fibrations. Then, as a consequence, the total spaces become homotopy equivalent. In this way, one may obtain elements in $\sS^{\TOP}(M_{m,n}^k)$ whose underlying manifold is $M_{m',n}^k$ for some other value of $m'$. Spherical fibrations have their own classifying spaces, and hence it is natural that this idea is investigated through them. At the level of classifying spaces, the map that corresponds to viewing an orthogonal sphere bundle as a spherical fibration is closely related to the classical $J$-homomorphism from homotopy theory. It has both an unstable version and a stable version, and both are relevant to us. We recall the necessary background; if needed, see \cite[Chapter 6]{LM24} for more details. 
	
	Consider, for $p \ge 1$:
	\[
	\SH(p):= \{f: S^{p-1} \rightarrow S^{p-1} \mid f \text{ is an orientation preserving homotopy equiv.} \}
	\]
	This is a topological monoid with respect to the composition of maps. We also consider $\SF (p) \subset \SH (p)$ the submonoid of the base point and orientation preserving homotopy equivalences. Their classifying spaces become the classifying spaces $S^{p-1}$-spherical fibrations and $S^{p-1}$-spherical fibrations equipped with a section, respectively.
	
	In particular, fiber homotopy equivalence classes of oriented $S^{p-1}$-fibrations over $S^{q}$ are in one-to-one correspondence with elements of $\pi_{q-1}(\SH(p))$. By restricting elements of $\SO(p)$ to $S^{p-1}$, we get a natural inclusion $\SO(p) \hookrightarrow \SH(p)$. On homotopy groups it induces a homomorphism which we denote $J^H \colon \pi_{q-1}(\SO(p)) \rightarrow \pi_{q-1}(\SH(p))$, since later it will be identified with the classical $J$-homomorphism. The next proposition is a direct consequence of Theorem 6.2 and Corollary 7.4 of \cite{DL59}.
	\begin{proposition} \label{prop:criterion-spherical-fibrations}
		Given two vector bundles $\xi_1$ and $\xi_2$ in $\pi_{q-1}(\SO(p))$. Their corresponding sphere bundles $S(\xi_{1})$ and $S(\xi_2)$ are fiber homotopy equivalent if and only if 
		\[
		J^H (\xi_1) = J^H (\xi_2) \in \pi_{q-1}(\SH(p)).
		\]	
	\end{proposition}
	
	One may also consider the inclusion $\SH (p) \subset \SF (p+1)$ obtained from suspension. By the result of James, see \cite[p.3]{James(1954)}, it induces an isomorphism 
	\begin{equation} \label{eqn:SH-p-versus-SF-p-plus-1}
		\pi_r (\SH(p)) \cong \pi_r (\SF (p+1)) \quad \textup{for} \quad r < 2p-4.    
	\end{equation}
	This is convenient, especially in view of the adjoint isomorphism 
	\begin{equation} \label{eqn:pi-m-SF-n-versus-pi-m-plus-n-of-S-n}
		\pi_m (\SF(n)) \cong \pi_{m+n-1} (S^{n-1});
	\end{equation}
	see \cite[(6.32)]{LM24}. 
	
	From this, it follows that we can replace in Diagram $(\ast)$ from \cite[page 182]{CCPS23} the bottom left entry $\pi_{8m-1} (S^{4m})$ by $\pi_{4m-1} (SH (4m))$, and use known knowledge about the middle column map, which contains the classical $J$-homomorphism, to find bundles that become fiber homotopy equivalent as spherical fibrations via Proposition \ref{prop:criterion-spherical-fibrations}. 
	
	We now recall some facts about the $J$-homomorphism, but we also recall some details about Diagram $(\ast)$ from \cite[page 182]{CCPS23}, since we have to make sure not only that we can replace the bottom left entry as described in the above paragraph, but also that the corresponding maps are compatible so that after the replacement the new square is commutative.
	
	The sources in the $J$-homomorphism are the homotopy groups of the (special) orthogonal groups. We start by recalling from Bott periodicity the stable versions ($n \geq 1$):
	\begin{equation}\label{eqn:htpy-groups-of-so}
		\pi_{n+1} (\BSO) \! \cong \! \pi_{n} (\SO) \! \cong \! \Z/2\Z, \Z/2\Z,0,\Z,0,0,0,\Z \; \textup{for} \; n \equiv 0,\ldots, 7 \; (\textup{mod} 8)
	\end{equation}
	
	Next, let us discuss the unstable version, keeping in mind the condition $k \geq 3$. We have already stated the calculation of $\pi_{4k-1} (\SO(4k))$ in \eqref{z2z}, but we need to provide more information. The isomorphism is obtained by considering two exact sequences involving $\pi_{4k-1} (\SO (4k))$. Namely, the first comes from the inclusion $\SO (4k-1) \hookrightarrow \SO (4k)$ and the second from $\SO (4k) \hookrightarrow \SO (4k+1)$, see \cite[p. 64-66]{Levine(1983)}. These produce two short exact sequences
	\begin{equation} \label{eqn:les-with-e}
		0 \rightarrow \pi_{4k-1} (\SO (4k-1)) \rightarrow \pi_{4k-1} (\SO (4k)) \xrightarrow{e} 2 \Z \rightarrow 0
	\end{equation}     
	\begin{equation} \label{eqn:les-with-S}
		0 \rightarrow \Z \xrightarrow{TS^{4k}} \pi_{4k-1} (\SO (4k)) \xrightarrow{S} \pi_{4k-1} (\SO (4k+1)) \rightarrow 0
	\end{equation}
	which both split and, by diagram chasing, we see that the map $(S,e)$ from \eqref{z2z} is an isomorphism. Moreover, the composition (which we also denote by $S$)
	\begin{equation} \label{eqn:comp-fro-unstable-so-4k-1-to-so-is-iso}
		\pi_{4k-1} (\SO (4k-1)) \rightarrow \pi_{4k-1} (\SO (4k)) \xrightarrow{S} \pi_{4k-1} (\SO (4k+1)) \cong \Z
	\end{equation}
	is an isomorphism. 
	
	We will also need the following straightforward reformulation. Using the canonical isomorphisms $\pi_{i-1}(\SO(j)) \cong \pi_{i}(\BSO(j))$ for all $i,j$ and $\pi_{i}(\BSO(j)) \cong \pi_{i}(\BSO(j+1))$ for $i<j$, we obtain the commutative square
	\begin{equation} \label{diag:dso-4k-to-bso}
		\xymatrix{
			\pi_{4k}(\BSO(4k)) \ar[r]^-{S} \ar[d]^{\cong}_{(S,e)} & \pi_{4k}(\BSO) \ar[d]^{\cong} \\
			\Z \oplus 2\Z \ar[r] & \Z
		}
	\end{equation}
	whose lower horizontal map is given by $(m,n) \mapsto m$. 
	
The unstable $J$-homomorphism is the map 
\[
J_{m,n} \colon \pi_{m} (\SO (n)) \rightarrow \pi_m (\SF(n+1)) \cong \pi_{m+n} (S^{n})
\]
defined as in \cite{GWW}. By increasing $n$, it stabilizes to 
\[
J_{m} \colon \pi_{m} (\SO) \rightarrow \pi_{m}^s.
\]
Hence, the commutative square of stabilization of the $J$-homomorphism (note that the composition in (\ref{eqn:comp-fro-unstable-so-4k-1-to-so-is-iso}) is onto) becomes
\begin{equation}\label{eqn:diagram-StJimage}
	\begin{tikzcd}
		\pi_{4k-1}(\SO(4k-1)) \arrow[d,"J_{4k-1,4k-1}"] \arrow[r] & \pi_{4k-1}(\SO(4k)) \arrow[d,"J_{4k-1,4k}"] \arrow[r] & \pi_{4k-1}(\SO) \arrow[d,"J_{4k-1}"] \\
		\pi_{4k-1}(\SF(4k)) \arrow[r] & \pi_{4k-1}(\SF(4k+1)) \arrow[r] & \pi_{4k-1}^{s}
	\end{tikzcd}
\end{equation}
where horizontal maps are stabilization homomorphisms. In both rows, the compositions from the first to the last are, in fact, isomorphisms. In the top row, this is \eqref{eqn:comp-fro-unstable-so-4k-1-to-so-is-iso}, and in the bottom row, we show this below in \ref{eqn:pi-8k-2-S-4k-1-to-pi-8k-S-4k-plus-1}. Granting this, we get
\begin{equation}\label{StJimage}
	\text{im}(J_{4k-1,4k-1}) = \text{im}(J_{4k-1}).
\end{equation}
In fact, we have the decomposition 
\begin{equation} \label{eqn:splitting-of-coker-ofJ}
	\pi_{4k-1}^{s} = \text{im}(J_{4k-1}) \oplus K(e^{A})    
\end{equation}
where $K(e^{A})$ is the kernel of Adams' $e$-invariant, see \cite[Theorem 1.1]{JA66}, which we denote by $e^{A}$ to distinguish it from the Euler class. The image of $J_{4k-1}$ has been determined by Adams and by \eqref{StJimage} we also get the same result for the image of $J_{4k-1,4k-1}$ (see \cite[Theorem 1.5]{JA66}):

\begin{proposition}[Adams] \label{prop:image-of-J}
	The subgroup $\text{im}(J_{4k-1})$ is a cyclic group of order $j_{k} =$ denominator$\left(\frac{B_k}{4k}\right)$ where $B_k$ is the $k$-th Bernoulli number. Therefore, we also have that the order of $\text{im}(J_{4k-1,4k-1}) = j_k$.
\end{proposition}

In order to fully utilize the information about the image of $J$ in our context, we need to connect the above with the homotopy theory of the space $\SH (p)$. Consider the fibration
\begin{center}
	$\SF(4k) \hookrightarrow \SH(4k) \xrightarrow{p} S^{4k-1}$
\end{center}		
where $p$ is the evaluation map. It induces a long exact sequence of homotopy groups. Keeping in mind $\pi_{4k-1} (\SH (4k)) \cong \pi_{4k-1} (\SF (4k+1))$ from~\eqref{eqn:SH-p-versus-SF-p-plus-1}, one can also consider the inclusion $\SF (4k+1) \hookrightarrow \SF (4k+2)$ and the associated long exact sequence of homotopy groups. By arguments from \cite[p. 1-3]{James(1954)}, these two exact sequences can be identified with the respective EHP sequences in suitable ranges. By the arguments described in Chapter 5.D about the EHP Sequence in \cite[p. 611]{AH02}, no localization for either parity is needed in these ranges. Then, arguing analogously to $\SO$, one obtains the isomorphism 
\begin{equation} \label{eqn:htpy-of-sh-4k}
	\pi_{4k-1} (\SH (4k)) \cong \pi_{4k-1} (\SF (4k+1)) \cong \pi_{8k-1} (S^{4k}) \xrightarrow[\cong]{S \oplus H} \pi_{4k-1}^s \oplus 2\Z 
\end{equation}
from the bottom row of (*) in \cite[p. 11]{CCPS23}. In addition, when studying the EHP sequence associated to the inclusion $\SF (4k) \hookrightarrow \SF (4k+1)$ one may observe that the induced map (which corresponds to the suspension homomorphism)
\[
\pi_{8k-2} (S^{4k-1}) \cong \pi_{4k-1} (\SF (4k)) \rightarrow \pi_{4k-1} (\SF (4k+1)) \cong \pi_{8k-1} (S^{4k})
\]
is injective by the following arguments. One can use that the map to the left of it in the EHP sequence is $\Z/2\Z \cong \pi_{8k} (S^{8k-1}) \rightarrow \pi_{8k-2} (S^{4k-1})$ and that it factors through $\pi_{4k-1} (\SO (4k-1)) \cong \Z$, and hence is zero. Alternatively, one may observe that this map is given by a suitable Whitehead product (see \cite[p. 1]{James(1954)} on the generator of $\Z/2\Z \cong \pi_{8k} (S^{8k-1})$, and this Whitehead product is zero by \cite[Theorem 1.1.2a]{Mahowald(1965)} and \cite[p. 482-486]{Whitehead(1978)}. Once we have this, a direct diagram chase shows that 
\begin{equation} \label{eqn:pi-8k-2-S-4k-1-to-pi-8k-S-4k-plus-1}
	\pi_{8k-2} (S^{4k-1}) \cong \pi_{4k-1} (\SF (4k)) \rightarrow \pi_{4k-1} (\SF (4k+2)) \cong \pi_{4k-1}^s
\end{equation}
is an isomorphism.

In Section~\ref{sec:realisation_results} we will consider the map of fibrations where all maps in the left square are inclusions
\begin{equation} \label{diag:map-of-two-fibrations}
	\xymatrix{
		\SO(4k-1) \ar[r] \ar[d] & \SO(4k) \ar[r]^{p} \ar[d] &  S^{4k-1} \ar[d]^{=} \\
		\SF(4k) \ar[r] & \SH(4k) \ar[r]^{p} &  S^{4k-1}
	}
\end{equation}

Using the isomorphisms \eqref{z2z} and \eqref{eqn:htpy-of-sh-4k}, on $\pi_{4k-1}$ the middle vertical map becomes the following:
\begin{equation} \label{eqn:the-J-homomorphism-we-need}
	(J_{4k-1} \oplus \id) \colon \Z \oplus \Z \cong \pi_{4k-1} (\SO (4k)) \rightarrow \pi_{4k-1} (\SH (4k)) \cong \pi_{4k-1}^s \oplus 2\Z.
\end{equation}

The above discussion summarizes useful data about the difference between orthogonal bundles and spherical fibrations via their classifying spaces. Next, it is necessary to connect these data to various other classifying spaces appearing in surgery theory. Homotopy groups of these spaces are organized in the well-known Kervaire-Milnor braid, see \cite[Section 12.7]{LM24} from which much information can be obtained.


\subsection{The Kervaire-Milnor braid and results of Brumfiel}\label{subsec:sKM-braid-and-Brumfiel-results}

The next important tool is the Kervaire-Milnor braid, see e.g. \cite[Sec 12.7]{LM24}, and Brumfiel's results from \cite{GB68} related to it. The braid is induced by the following commutative diagram of infinite loop spaces, where all rows and columns are homotopy fibration sequences:
\[
\xymatrix{
	\TOPmodO \ar[r] \ar[d]_{=} & \GmodO \ar[d] \ar[r] & \GmodTOP \ar[d]  \\
	\TOPmodO \ar[r] & \BO \ar[r] \ar[d] & \BTOP \ar[d] \\
	& \BG \ar[r]_{=} & \BG 
}
\]
We are interested in the following piece of this braid, which is obtained by applying homotopy groups to the above diagram and using the isomorphisms $\pi_i (\BG) \cong \pi_i^s $, $\pi_i (\GmodTOP) \cong L_i (\Z)$, and $\pi_i (\TOPmodO) \cong \Theta_i$ for all $i \geq 5$ (see \cite[Theorem 12.59]{LM24}). In the diagram $k \geq 3$: 
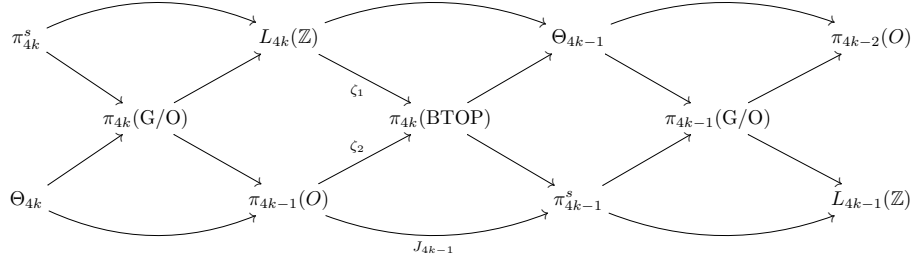
\[
\scalebox{0.75}{
	\xymatrix{
		\pi_{4k}^{s} \ar[dr] \ar@/u\curv/[rr] && L_{4k}(\Z) \ar[dr]_{\zeta_{1}}\ar@/u\curv/[rr] &&\Theta_{4k-1} \ar[dr]\ar@/u\curv/[rr]&& \pi_{4k-2}(O)    \\
		& \pi_{4k}(\GmodO) \ar[dr] \ar[ur] && \pi_{4k}(\BTOP) \ar[dr] \ar[ur]         && \pi_{4k-1}(\GmodO) \ar[dr] \ar[ur] &     \\ 
		\Theta_{4k} \ar[ur] \ar@/d\curv/[rr] && \pi_{4k-1}(O) \ar[ur]^{\zeta_{2}} \ar@/d\curv/_{J_{4k-1}}[rr]
		&&\pi_{4k-1}^{s} \ar[ur]\ar@/d\curv/[rr] &&
		L_{4k-1}(\Z)
	}
}
\]

\begin{remark}
	We are using the $\TOP$-version rather than the $\PL$-version of the braid. The homotopy fiber $\TOP/\PL$ of the canonical map $\BPL \rightarrow \BTOP$ has the homotopy type of the Eilenberg-MacLane space $K(\Z/2\Z,3)$. The long exact sequence of homotopy groups implies that the homotopy groups of $\BPL$ and $\BTOP$ coincide in higher dimensions. Since we are working in dimension $\ge 5$, we can replace $\pi_{*}(\BTOP)$ with $\pi_{*}(\BPL)$. 
\end{remark}
The groups $\pi_{4k}(\BTOP) \cong \pi_{4k}(\BPL)$ are known to be (see \cite[Theorem 1.4, Remark 4.9]{GB68}): 
\begin{center}
	$\pi_{4k}(\BTOP) \cong \pi_{4k}(\BPL) \cong \Z \oplus \pi_{4k-1}^{s}/\text{im}(J_{4k-1})'$ for $k \ge 3$
\end{center}	
where $\text{im}(J_{i})' \subset \text{im}(J_{i})$ denotes the subgroup of $\text{im}(J_i)$ containing elements of odd order. We shall use the notation $T:= \pi_{i}^s/\text{im}(J_{i})'$ so that $\pi_{4k}(\BTOP) \cong \Z \oplus T$. By observing that one of the sequences in the braid is the surgery exact sequence for homotopy spheres (or by \cite[page 18]{DC10}) we have 
\begin{equation} \label{eqn:pi_4k-of-G-mod-O}
	\pi_{4k}(\GmodO) \cong \Z \oplus \Theta_{4k}.    
\end{equation}
After replacing the known values, the braid becomes:
\vspace{0.5cm}
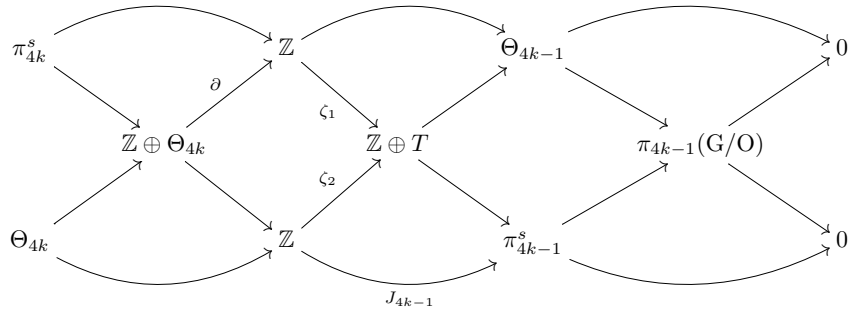
\[
\scalebox{0.9}{
	\xymatrix{
		\pi_{4k}^{s} \ar[dr] \ar@/u\curv/[rr] && \Z \ar[dr]_{\zeta_{1}}\ar@/u\curv/[rr] &&\Theta_{4k-1} \ar[dr]\ar@/u\curv/[rr]&& 0    \\
		& \Z\oplus \Theta_{4k} \ar[dr] \ar[ur]^{\partial} && \Z \oplus T \ar[dr] \ar[ur] && \pi_{4k-1}(\GmodO) \ar[dr] \ar[ur] &     \\ 
		\Theta_{4k} \ar[ur] \ar@/d\curv/[rr] && \Z \ar[ur]^{\zeta_{2}} \ar@/d\curv/_{J_{4k-1}}[rr] &&\pi_{4k-1}^{s} \ar[ur]\ar@/d\curv/[rr] && 0
	}
}
\]

The maps $\zeta_{1}$ and $\zeta_{2}$ in the braid will play an important role in obtaining our results. They have been determined in \cite{GB68}. We shall briefly discuss it here. First, let us borrow some notations from the same article:
\begin{equation}\label{equation:nota}
	\begin{split}
		j_k &= \text{denominator}\left(\frac{B_k}{4k}\right); \\ j_k'& = \text{largest odd divisor of }j_k;\\
		\theta_k & = \text{numerator}\left(\frac{B_k}{4k}\right)\cdot a_k \cdot 2^{2k-2}\cdot(2^{2k-1}-1); \\ 
		\theta_k' & = \text{ largest odd divisor of }\theta_k;\\
		\text{define } b_k \text{ by }2^{b_k} & = (\theta_{k} \cdot j_k')/(\theta_k' \cdot j_k)
	\end{split}
\end{equation}
where $B_k$ is $k$-th Bernoulli number (topological indexing) and $a_k$ is $1$ or $2$ according to $k$ is even or odd. It is well known that numerator$\left(\frac{B_k}{4k}\right)$ is always odd and also gcd$(j_k',\theta_k') = 1$ (see \cite[Page 305]{GB68}). Note that 
\begin{equation}\label{equation:alter}
	j_{k}= |\text{im} (J_{4k-1})|, \theta_{k} = |bP_{4k}|.
\end{equation}

\textbf{The reader needs to be careful here with the notations $\theta_k$ and $\Theta_k$. The former is an integer defined above, and the latter is the well known Kervaire-Milnor group of $h$-cobordism classes of homotopy spheres}.\\

Part 1. of the following lemma is an adaptation of \cite[Lemma 4.5]{GB68}, which determines the maps $\zeta_{1}$ and $\zeta_{2}$, and part 2. follows from the work of Kervaire and Milnor; see \cite[theorem 12.49]{LM24}:
\begin{lemma} \label{lem:Brumfiel-result-about-zeta_one-two}
	Let $k\ge 3$. 
	\begin{enumerate}
		\item The maps
		\begin{center}
			$\zeta_1: \Z \cong L_{4k}(\Z) \rightarrow \pi_{4k}(\BTOP) \cong \Z \oplus T$\\
			$\zeta_2: \Z \cong \pi_{4k}(\BO) \rightarrow \pi_{4k}(\BTOP) \cong \Z \oplus T$ 
		\end{center}
		of the Kervaire-Milnor Braid are given by
		\begin{center}
			$\zeta_{1}(1) = (j_k'\cdot \beta, t_1)$\\
			$\zeta_{2}(1) = (2^{b_{k}}\cdot \theta_k' \cdot \beta, t_2)$
		\end{center}	
		with $\beta$ a generator of the infinite component of $\pi_{4k}(\BTOP)$ and $t_{1}, t_{2} \in T$.
		\item The map $\partial \colon \Z \oplus \Theta_{4k} \cong \pi_{4k} (\GmodO) \rightarrow \pi_{4k} (\GmodTOP) \cong \Z$ in the Kervaire-Milnor braid is given by multiplication by $\theta_k$ on the $\Z$-summand.
	\end{enumerate}
\end{lemma}


\section{Realisation results}\label{sec:realisation_results}



Now we move towards proving our next results. We start with the following three Propositions below. Similarly as done in \cite{RM24}, Propositions \ref{foa} and \ref{gen2} are adaptations of Lemmas 5.2 and 5.3 respectively from \cite{CE03}.
\begin{proposition}\label{foa}
	Let $n \ge 0$ and $m,t \in \Z$ be arbitrary. There exists a fibre homotopy equivalence $f_{t}: M_{m + j_k \cdot t, n}^k \rightarrow M_{m,n}^k$.
\end{proposition}
\begin{proof}
	Discussion about the stable and unstable versions of the $J$-homomorphism in Section \ref{subsec:sphere-bdles-and-spherical-fibrations} shows that Diagram \eqref{diag:map-of-two-fibrations} induces the following commutative diagram on homotopy groups: 
	\begin{center}
		\begin{tikzcd}[row sep=2em,column sep=2em]
			0 \arrow[r] & \Z \arrow[r] \arrow[d,"\cong"] & \Z \oplus 2\Z \arrow[r] \arrow[d,"\cong"] & 2\Z \arrow[r] \arrow[d,"\cong"] & 0  \\
			0 \arrow[r] & \pi_{4k-1}(\SO(4k-1)) \arrow[r] \arrow[d,"J_{4k-1,4k-1}"] & \pi_{4k-1}(\SO(4k)) \arrow[r,"p'_{*}"] \arrow[d,"J^H"] & 2\pi_{4k-1}(S^{4k-1}) \arrow[r] \arrow[d,"Id"] & 0 \\
			0 \arrow[r] & \pi_{4k-1}(\SF(4k)) \arrow[r] \arrow[d,"\cong"] & \pi_{4k-1}(\SH(4k)) \arrow[r,"p_{*}"] \arrow[d, "\cong"] & 2\pi_{4k-1}(S^{4k-1}) \arrow[r] \arrow[d,"\cong"] & 0 \\
			0 \arrow[r] & \pi_{4k-1}^s \arrow[r] & \pi_{4k-1}^s \oplus 2\Z \arrow[r] & 2\Z \arrow[r] &0
		\end{tikzcd}
	\end{center}
	By Proposition \ref{prop:image-of-J} we have $\im (J_{4k-1,4k-1}) \cong \Z/j_k \Z$. Upon this identification and using the indicated isomorphisms, the middle vertical map $J^H$ is given by the formula
	\[
	J^H (m,2n) = (m \! \mod j_k,2n).    
	\]
	The result now follows from Proposition~\ref{prop:criterion-spherical-fibrations}.
\end{proof}

\begin{proposition}\label{gendiff}
	The composition
	\begin{center}
		$\pi_{4k-1}(\SO(4k)) \cong \pi_{4k}(\BSO(4k)) \xrightarrow{i_*} \pi_{4k}(\BSO) \xrightarrow{\zeta_2}\pi_{4k}(\BTOP) \xrightarrow{\cong} \Z \oplus T$
	\end{center}	
	sends the difference of the equivalence classes of vector bundles $[\gamma_{m+j_{k}\cdot t,n }^k] - [\gamma_{m,n}^k]$ to $(\theta_k \cdot j_k' \cdot t, t_1)$ where $t_1 \in T$ is some element of the torsion part.
\end{proposition}

\begin{proof}
	Using \eqref{diag:dso-4k-to-bso} the composition is described as follows
	\begin{center}
		$[\gamma_{m+ j_{k} \cdot t,n}^k] - [\gamma_{m,n}^k] \mapsto (j_k \cdot t, 0) \mapsto j_{k} \cdot t \mapsto ((2^{b_k} \cdot \theta_k' \cdot j_k) \cdot t, t_1) = (\theta_k \cdot j_k' \cdot t, t_1)$
	\end{center}
	where last equality is obtained after replacing the value $2^{b_{k}} = (\theta_{k} \cdot j_k')/(\theta_k' \cdot j_k)$ from (\ref{equation:nota}).
\end{proof}

The proof of the next proposition requires the following observation about bundle theory. Given a  vector bundle $\xi$ over a compact manifold $B$, we denote by $-\xi$, its stable inverse,  by $\tau_{B}$ the stable tangent bundle of $B$ and by $\nu_{B} = - \tau_{B}$ its stable normal bundle. As in~\cite[Fact 3.1]{CE03}, we have the following bundle isomorphisms:
\begin{equation}\label{stable}
	\nu_{S(\xi)} \cong p^{*}_{\xi}(\nu_{B} \oplus -\xi) \text{ and } \tau_{S(\xi)} \cong p^{*}_{\xi}(\tau_{B} \oplus \xi).	
\end{equation}
An analogous statement holds for the associated disk bundle $D(\xi)$ instead of the associated sphere bundles $S(\xi)$. 

\begin{proposition}\label{gen2}
	The fibre homotopy equivalence $f_{t}: M_{m + j_{k} \cdot t, n}^k \rightarrow M_{m,n}^k$ has normal invariant $\eta(f_t) = {\theta_k \cdot t} \in \mathcal{N}^{\TOP}(M_{m,n}^k) \cong \Z/2n\Z$.
\end{proposition}

\begin{proof}
	For a given map of sphere bundles $f : S(\gamma) \rightarrow S(\chi)$ the mapping cone of $f$ defines a map of disk bundles $F : D(\gamma) \rightarrow D(\chi)$ which for $0 \leq r \leq 1$ and $r \cdot v \in D(\gamma)$ takes the value $F(r \cdot v) = r \cdot f(v)$.  Also the pair $(F_{t}, f_{t}) : (W_{m+j_k \cdot t,n}^k, M_{m+j_k \cdot t,n}^k) \rightarrow (W_{m,n}^k, M_{m,n}^k)$ is a fibre homotopy equivalence. Diagram (\ref{bdry}) implies $\eta(f_{t}) = i^{*}\eta(F_t)$ and Diagram~\eqref{equation:ses} implies it is enough to prove that $\eta(F_t) \in \mathcal{N}^{\TOP}(W_{m,n}^k)$ takes on the value $\theta_k \cdot t \in \Z$. The inclusion $i_{m,n}:S^{4k} \rightarrow W_{m,n}^k$, as zero section, is homotopy equivalence. Recall the fibration $\GmodTOP \xhookrightarrow{j} \BTOP \rightarrow BG$ and consider the following diagram:
	\begin{center}
		\begin{tikzcd}[row sep=1em,column sep=1em]
			\mathcal{N}^{\TOP}(W_{m,n}^k) \arrow[d] \arrow[r,"\cong"] & \left[W_{m,n}^k,\GmodTOP \right] \arrow[d,"i^{*}_{m,n}"] \arrow[r,"j_*"] & \left[W_{m,n}^k,\BTOP \right] \arrow[d,"i^{*}_{m,n}"] \\
			\mathcal{N}^{\TOP}(S^{4k}) \arrow[r,"\cong"]  & \left[S^{4k},\GmodTOP \right] \arrow[r,"j_*"] & \left[S^{4k},\BTOP \right]
		\end{tikzcd}
	\end{center}
	Here $i_{m,n}^*$ is bijective, being induced from a homotopy equivalence. From the Kervaire-Milnor braid we observe that $j_{*} : \pi_{4k}(\GmodTOP) \rightarrow \pi_{4k}(\BTOP)$ is the map $\zeta_1 : \Z \rightarrow \Z \oplus T$, which is given by $x \mapsto (j_k' \cdot x, t_1 \cdot x)$. For a compact space $Y$, the group $[Y, \BTOP]$ may be regarded as formal differences of stable topological microbundles over $Y$. Hence for $Y = W_{m,n}^k$, which is a smooth manifold, its stable normal vector bundle is also its stable normal microbundle and we have 
	\begin{center}
		$j_{*}(\eta(F_{t})) = \nu(W_{m,n}^k) - F^{-1*}_{t}(\nu(W_{m+j_{k} \cdot t,n}^k))$.
	\end{center}
	Fact (\ref{stable}) for disk bundles implies that 	
	\begin{center}
		$\nu(W_{m,n}^k) = p^{*}_{m,n}(\nu_{S^{4k}} \oplus -\gamma_{m,n}^k) = p^{*}_{m,n}(-\gamma_{m,n}^k)$.
	\end{center}
	Since $F_{t}$ commutes with $p_{m,n}$ and $p_{m+j_k \cdot t,n}$, we may choose $F^{-1}_{t}$ to commute up to homotopy. Therefore
	\begin{align*}
		i^{*}_{m,n}(j_{*}(\eta(F_{t}))) & =   i^{*}_{m,n}(\nu(W_{m,n}^k) - F^{-1*}_{t}(\nu(W_{m+j_k \cdot t,n}^k))) \\
		&= i^{*}_{m,n}(p^{*}_{m,n}(-\gamma_{m,n}^k) - F^{-1*}_{t}(p^{*}_{m+j_k \cdot t,n}(-\gamma_{m+j_k \cdot t,n}^k)))\\
		&= i^{*}_{m,n}(p^{*}_{m,n}(-\gamma_{m,n}^k+\gamma_{m+j_k \cdot t,n}^k)))\\
		&= \gamma_{m+j_k \cdot t,n}^k - \gamma_{m,n}^k
	\end{align*}	
	By Proposition \ref{gendiff} this difference is sent to the pair $(\theta_k \cdot j_k'\cdot t,t_1) \in \Z\oplus T$. This gives us $\eta(F_{t}) = \frac{\theta_k \cdot j_k' \cdot t}{j_k'} = \theta_k \cdot t \in \mathcal{N}^{\TOP}(W_{m,n}^k)$ hence $\eta(f_{t}) = \theta_k \cdot t \in \mathcal{N}^{\TOP}(M_{m,n}^k)$.
\end{proof}

Using above Propositions we define a map
\begin{equation}\label{equation:alipha}
	\alpha_{m,n}^k:A_{m,n}^k \rightarrow S^{\TOP}(M_{m,n}^k)
\end{equation}
where $A_{m,n}^k=\{(m', n)\hspace{.2cm}|\hspace{.2cm} m'=m+j_k\cdot t \text{ where }t \in \Z\} \subset \pi_{4k-1}(\SO(4k)) \cong \Z \oplus 2\Z$. It is given by sending $(m',n)$ to the element represented by the homotopy equivalence $f_{t} \colon M_{m',n}^k \rightarrow M_{m,n}^k$. We obtain the following:

\begin{theorem}\label{thm:alpha}
	The image of the map
	$\alpha_{m,n}^k \colon A_{m,n}^k \rightarrow \mathcal{S}^{\TOP} (M_{m,n}^k) \cong \Z/2n\Z$ is (gcd$(n,\theta_k)) \cdot \mathbb{Z}/2n\Z$.
\end{theorem}

\begin{proof}
	Proposition \ref{foa} implies map $\alpha_{m,n}^k$ is well defined and Proposition \ref{gen2} provides it image.
\end{proof}	

Note that by construction, all elements in the image of $\alpha_{m,n}^{k}$ have a smooth source manifold and hence are contained in the image of the forgetful map $F_{M_{m,n}^k} \colon \mathcal{S}^{\DIFF}(M_{m,n}^k) \rightarrow \mathcal{S}^{\TOP}(M_{m,n}^k)$.

Theorem \ref{thm:alpha} immediately implies Corollary~\ref{cor:realization-of-ni-when-n-theta-coprime} stated in Section~\ref{sec:leftovers}.


\section{The forgetful map}\label{sec:forgetful_map}


In this section, we shall prove Theorem \ref{thm:the-forgetful-map-on-str-sets}. The key part will be the following theorem. The rest will be dealt with via the surgery machine at the end of the section. We assume $k \geq 3$.

\begin{theorem}\label{fn} 
	The cokernel of the forgetful map 
	\[
	\chi:\mathcal{N}^{\DIFF}(M_{m,n}^k) \rightarrow \mathcal{N}^{\TOP}(M_{m,n}^k)
	\]
	is isomorphic to $\Z/g\Z$ where $g$=gcd$(e,\theta_k)$.
\end{theorem}
The stable $J$-homomorphism, $J_{4k-1}:\pi_{4k-1}(\SO)\rightarrow \pi_{4k-1}^s$, is used in the proof. This homomorphism is almost always not surjective, but there are a few instances when it happens to be surjective, for example, from \cite{IWX20} one can tell this is the case for $4k-1 = 7,11,27 \text{ and } 43$. From a certain point, which is clearly marked, we split the proof into two versions depending on whether $J_{4k-1}$ is surjective or not, since this influences the level of difficulty. 

\noindent \textbf{Proof of Theorem~\ref{fn}:} 
We first look at some homotopical properties of the total space $M_{m,n}^k$. Since $k,m,n$ are fixed throughout the proof, in order to avoid unnecessary clutter in the  notation, we abbreviate $M = M_{m,n}^{k}$. The total space has the following cell structure (see \cite[Equation 3.3]{JW54}):
\begin{equation}\label{cell}
	M = S^{4k-1} \underset{\alpha}{\cup} e^{4k} \underset{\beta}{\cup} e^{8k-1}	
\end{equation}
where  $\beta:S^{8k-2} \rightarrow S^{4k-1} \underset{\alpha}{\cup} e^{4k}$ and $\alpha = p_{*}(\gamma_{m,n}^k)$ are induced from the projection $p: \SO(4k) \rightarrow S^{4k-1}$. In general, $p_*$ sends an $\R^q$ bundle over $S^p$, say $\xi$, to $O_{\xi}$, which is an obstruction against finding a section. If $q=p+1$, then this obstruction is just the Euler number, $e$, of the bundle. Therefore, $\alpha = p_*(\gamma_{m,n}^k)=e=2n$; \cite[p. 64]{Levine(1983)}. 

Let us denote by $Y = S^{4k-1} \underset{\alpha}{\cup} e^{4k}$. From \cite[Equation 3.7]{JW54} and \cite[Equation 5.1 and page 5]{JW55}, it follows that after collapsing the fiber, the total space $M$ is stably homotopy equivalent to the wedge sum
\begin{equation}\label{stableeq}
	M/S^{4k-1} \simeq_s S^{4k} \vee S^{8k-1}.	
\end{equation} 
Namely, since $Y/S^{4k-1} \cong S^{4k}$ the composition $\overline{\beta}: S^{8k-2} \xrightarrow{\beta} Y \xrightarrow{collapse}S^{4k}$ satisfies 
\begin{center}
	$\overline{\beta} \simeq_s 0$ by \cite[Equation 5.1]{JW55}
\end{center}
which implies (\ref{stableeq}).

Since in \eqref{stableeq} we have a stable homotopy equivalence, we obtain, for some $p$, a cofibration sequence of the form:
\begin{multline} \label{mcof}
	\Sigma^{p} S^{4k-1} \cong S^{4k-1+p} \rightarrow \Sigma^{p} M \rightarrow \Sigma^{p}(S^{4k-1} \vee S^{8k-1}) \cong S^{4k-1+p} \vee S^{8k-1+p} \rightarrow \\
	\rightarrow \Sigma^{p+1} S^{4k-1} \cong S^{4k+p} \rightarrow \Sigma^{p+1} M \rightarrow \cdots
\end{multline}

Given a spectrum (or an infinite loop space) $K$ and a space $X$, we may consider $K$-cohomology: 
\begin{equation}\label{RH}
	K^{q}(X) = [X,K]_{-q} = \text{colim}_{p}[\Sigma^{p}X,K_{q+p}].
\end{equation}
From \cite[Part 3, Proposition 6.1]{JA73} and \cite[Theorem 4.58]{AH02} we have that the functors $X \mapsto K^{q} (X) = [X, K]_{-q}$, $q \in \Z$, define a reduced cohomology theory on the category of base-pointed $CW$-complexes and base point preserving maps. 

The spaces $\GmodTOP$, $\GmodO$, $\G$ are infinite loop spaces \footnote{See also Lemma~\ref{lem:group-structures-on-ni-of-E} below.} and hence each gives us a spectrum. Applying the contravariant functors $[-,\GmodO]$ and $[-,\GmodTOP]$ to cofibration (\ref{mcof}) and using suspension and de-suspension isomorphisms in the respective $K$-cohomology, we obtain a commutative diagram as follows:  
\begin{equation}\label{mreln}
	\begin{tikzcd}[column sep=.7em, font=\small]
		\arrow[r] & \left[S^{4k},\GmodO\right] \arrow[r,"r"] \arrow[d] & \left[S^{4k} \vee S^{8k-1},\GmodO\right] \arrow[r,"r_1"] \arrow[d] & \left[M,\GmodO\right] \arrow[r,"r_2"] \arrow[d, "\chi"] & {} \\
		\arrow[r] & \left[S^{4k},\GmodTOP\right] \arrow[r] & \left[S^{4k} \vee S^{8k-1},\GmodTOP\right] \arrow[r] & \left[M,\GmodTOP\right] \arrow[r] & {} \\
		{} \arrow[r,"r_2"] & \left[S^{4k-1},\GmodO\right] \arrow[r,"e\oplus D_{*}"] \arrow[d] & \left[S^{4k-1}\vee S^{8k-2}, \GmodO\right] \arrow[d] \\
		{} \arrow[r] & \left[S^{4k-1},\GmodTOP\right] \arrow[r] & \left[S^{4k-1}\vee S^{8k-2},\GmodTOP\right]
	\end{tikzcd}
\end{equation}

We now study this diagram in detail. Topological surgery tells us that $[S^{n},\GmodTOP] \cong L_{n}(\Z)$ and we have already calculated the normal invariants set $\mathcal{N}^{\TOP}(M) \cong [M,\GmodTOP] \cong \Z/e\Z$ in Theorem \ref{thm:top-str-set}. We also obtain an improvement:

\begin{lemma}\label{lem:group-structures-on-ni-of-E}
	The two group structures on $\mathcal{N}^{\TOP}(M) \cong [M,\GmodTOP]$ coming from the two $H$-space structures on $\GmodTOP$ agree.
\end{lemma}

\begin{proof}
	The two group structures on $[S^{n},\GmodTOP]$ agree for any $n$ since both of them agree with the group structure obtained from the co-$H$-spaces structure of $S^n$. By the above sequence, $[M,\GmodTOP]$ is the quotient of a homomorphism $[S^{4k},\GmodTOP] \rightarrow [S^{4k},\GmodTOP]$, which proves the claim.
\end{proof}

By \eqref{eqn:pi_4k-of-G-mod-O} we have $[S^{4k},\GmodO] \cong \Z \oplus \Theta_{4k}$. By part 2. of Lemma~\ref{lem:Brumfiel-result-about-zeta_one-two}, the first two vertical maps are given as multiplication by $\theta_k$, which is defined in Notation \eqref{equation:nota}. The Kervaire-Milnor braid also tells us that $\pi_{4k-1}(\GmodO)$ is the cokernel of $J_{4k-1}$. Let us denote it by $c(J_{4k-1})$.

We need to discuss now the last map of the top row in (\ref{mreln}). It is given by the direct sum $e\oplus D_*$ where  $D_*$ is induced by a stable map $S^{8k-2} \rightarrow S^{4k-1}$ and $e$ is the Euler number by the following argument. 

Consider the cofibration sequence (\ref{mcof}). As indicated below \eqref{cell}, the first component of the third map $S^{4k}\vee S^{8k-1} \rightarrow \Sigma S^{4k-1} \cong S^{4k}$ is given as multiplication by $e$, i.e., $S^{4k} \xrightarrow{\cdot e} S^{4k}$, which, being a multiplication, remains the same after suspensions or de-suspensions. Hence, the first component of the last map in the top row of \eqref{mreln} is given as multiplication by $e$.

After replacing known values the above diagram becomes:
\begin{equation}\label{mreln1}
	\begin{tikzcd}[column sep=.7em, font=\small]
		\arrow[r] & \Z \oplus \Theta_{4k} \arrow[r,"r"] \arrow[d,"\theta_k"] & \Z \oplus \Theta_{4k}\oplus c(J_{8k-1}) \arrow[r,"r_1"] \arrow[d,"\theta_k"] & \left[M,\GmodO\right] \arrow[r,"r_2"] \arrow[d, "\chi"] & c(J_{4k-1}) \arrow[r,"e\oplus D_{*}"] \arrow[d] & c(J_{4k-1})\oplus \pi_{8k-2}(\GmodO) \arrow[d] \\
		\arrow[r] & \Z \arrow[r,"\cdot e"] & \Z \oplus 0 \arrow[r,"p"] & \Z/e\Z \arrow[r] & 0 \arrow[r] & 0 \oplus \Z/2\Z
	\end{tikzcd}
\end{equation}

\noindent \textbf{The case when $J_{4k-1}$ is surjective:}


\noindent In this case, the above diagram simplifies to: 

\ 

\begin{tikzcd}
	& \Z \oplus \Theta_{4k} \arrow[r,"r"] \arrow[d,"\theta_k"] & \Z \oplus \Theta_{4k}\oplus c(J_{8k-1}) \arrow[r,"r_1"] \arrow[d,"\theta_k"] & \left[M,\GmodO\right] \arrow[r,"r_2"] \arrow[d, "\chi"] & 0 \\
	0 \arrow[r] & \Z \arrow[r,"\cdot e"] & \Z \arrow[r,"p"] & \Z/e\Z 
\end{tikzcd}\\

The snake lemma gives us the following exact sequence of kernels and cokernels of vertical maps
\[
\Theta_{4k} \rightarrow \Theta_{4k} \oplus c(J_{8k-1}) \rightarrow K \xrightarrow{\delta} \Z/\theta_k\cdot \Z \xrightarrow{\cdot e} \Z/\theta_k \cdot \Z \rightarrow K' \rightarrow 0
\]
where $K = \text{kernel}(\chi)$, $K'= \text{cokernel}(\chi)$ and $\delta$ is the connecting homomorphism. By the first isomorphism theorem $K' \cong \Z/g\Z$, where $g = \text{gcd}(e,\theta_{k})$. This proves Theorem \ref{fn} when the $J_{4k-1}$-homomorphism is surjective. 

\


\noindent \textbf{The case when $J_{4k-1}$ is not surjective:} The proof in the previous case was obtained by studying a map on the normal invariants induced by the map $\GmodO \rightarrow \GmodTOP$, which occurs in the Kervaire-Milnor braid. The idea in this case is to employ some of the remaining maps in that braid to obtain more information, namely the part:
\[
\xymatrix{
	G \ar[r] \ar[d] & \GmodO \ar[r] \ar[d] & \BO \ar[d] \\
	G \ar[r] & \GmodTOP \ar[r] & \BTOP 
}
\]
It induces the following commutative diagram in which the map $\chi$ also fits in:
\[
\begin{tikzcd}
	\cdots \arrow[r] & \left[M,G\right] \arrow[r,"\delta"] \arrow[d, "="] & \left[M, \GmodO \right] \arrow[r] \arrow[d, "\chi"] & \left[M,\BO\right] \arrow[r] \arrow[d] & \left[M,BG\right] \arrow[d,"="] \arrow[r]&{\cdots} \\
	\cdots \arrow[r] & \left[M,G\right] \arrow[r] & \left[M,\GmodTOP\right] \arrow[r] & \left[M,\BTOP\right] \arrow[r] & \left[M,BG\right] \arrow[r]&{\cdots}
\end{tikzcd}
\]
If we denote by $R$ the cokernel of the map $r$ in the diagram (\ref{mreln}) then by combining it with above diagram we get the another commutative diagram:
\begin{equation}\label{chidisc}
	\begin{tikzcd}
		{} & \left[M,G\right] \arrow[d, "\delta"] \arrow[dd, bend left = 60, near start, "\chi'"] & {} \\
		R \arrow[r, "\overline{r}"] \arrow[dr,"\chi''"]  & \left[M,\GmodO\right] \arrow[r, "r_2"] \arrow[d,"\chi"] & \left[S^{4k-1},\GmodO\right] \\
		{} & \left[M,\GmodTOP\right] & {}
	\end{tikzcd}
\end{equation}
where $\chi'' = \chi \circ \overline{r}$ and $\chi' = \chi \circ \delta$. The horizontal row is exact and $\overline{r}$ is injective. We shall prove the following: 
\begin{lemma}\label{image}
	The image of the map $\chi$ is the group $\text{im}(\chi) = \langle \text{im}(\chi'),\text{im}(\chi'')\rangle$ generated by images of $\chi'$ and $\chi''$.
\end{lemma}
\begin{proof}
	It is easy to see that 
	\begin{center}
		$\langle \text{im}(\chi'),\text{im}(\chi'')\rangle \subseteq \text{im}(\chi)$.
	\end{center}
	Now let us prove inclusion the other way around.
	
	Let $y \in [M,\GmodO]$. Assume for a moment that there exists some $z \in [M,G]$ such that $r_2(\delta(z)) = r_2(y)$ in $[S^{4k-1},\GmodO]= \pi_{4k-1}(\GmodO) = c(J_{4k-1})$. Then we have:
	\begin{align*}
		r_2(\delta(z)-y) = r_2(\delta(z)) - r_2(y) = 0 & \implies \delta(z)-y = \overline{r}(x) \text{ for some } x \in R\\
		& \implies y = \delta(z) - \overline{r}(x)\\
		& \implies \chi(y) = \chi(\delta(z)) - \chi(\overline{r}(x)) = \chi'(z)-\chi''(x).
	\end{align*}
	This proves the other inclusion.
	
	It remains to show the existence of $z \in [M,G]$ that was assumed in the paragraph above. 
	
	Consider the following sequence of maps in which two consecutive arrows form a fibration up to homotopy: 
	\begin{center}
		$\cdots \rightarrow \G \rightarrow \GmodO \rightarrow \BO \rightarrow BG \rightarrow \cdots$
	\end{center} 
	After applying the covariant functors $[S^{4k-1},-]$ and $[S^{4k-1}\vee S^{8k-2},-]$, we get a commutative diagram of the long exact sequences that has a part looking like
	\[
	\begin{tikzcd}
		{[S^{4k-1},\Orthogroup]} \arrow[r,"e \oplus D_*"] \arrow[d,"J_{4k-1}"] & {[S^{4k-1},\Orthogroup] \oplus [S^{8k-2},\Orthogroup]} \arrow[d,"J_{4k-1}\oplus s"] \\
		{[S^{4k-1},\G]} \arrow[r,"e \oplus D_*"] \arrow[d,"f"] & {[S^{4k-1},\G] \oplus [S^{8k-2},\G]} \arrow[d,"g"] \\
		{[S^{4k-1},\GmodO]} \arrow[r,"e \oplus D_*"] \arrow[d,"f_1"] & {[S^{4k-1},\GmodO] \oplus [S^{8k-2},\GmodO]} \arrow[d,"g_1"] \\
		{[S^{4k-1},\BO]} \arrow[r,"e \oplus D_*"] & {[S^{4k-1},\BO] \oplus [S^{8k-2},\BO]}
	\end{tikzcd}
	\]
	Note that all the horizontal maps are different, but they are defined in the same way, so we are abusing the notation. After replacing known terms, see \eqref{eqn:htpy-groups-of-so}, the diagram becomes 
	\[
	\begin{tikzcd}
		\Z \arrow[r,"e \oplus D_*"] \arrow[d,"J_{4k-1}"] & \Z \oplus 0 \arrow[d,"J_{4k-1}\oplus s"] \\
		{[S^{4k-1},\G]} \arrow[r,"e \oplus D_*"] \arrow[d,"f"] & {[S^{4k-1},\G] \oplus [S^{8k-2},\G]} \arrow[d,"g = f \oplus v"] \\
		{[S^{4k-1},\GmodO]} \arrow[r,"e \oplus D_*"] \arrow[d,"f_1"] & {[S^{4k-1},\GmodO] \oplus [S^{8k-2},\GmodO]} \arrow[d,"g_1"] \\
		0 \arrow[r,"e \oplus D_*"]           & 0 \oplus 0
	\end{tikzcd}
	\]
	By exactness, we can see that $v$ in the second column is an isomorphism. Also, in the first column, since $[S^{4k-1},\GmodO]$ is the cokernel of $J_{4k-1}$, we shall have a splitting in the opposite direction of $f$, say $f'$ (the splitting exists because of \eqref{eqn:splitting-of-coker-ofJ}). We use $f'$ to define a splitting for the second column $g'= f' \oplus v^{-1}$. Therefore, the middle square becomes:
	\[
	\begin{tikzcd}
		{[S^{4k-1},G]} \arrow[r,"e\oplus D_*"] \arrow[d,"f"] & {[S^{4k-1},G] \oplus [S^{8k-2},G]} \arrow[d,"g = f\oplus v"] \\
		{[S^{4k-1},\GmodO]} \arrow[r,"e\oplus D_*"] \arrow[u, bend left,"f'"] & {[S^{4k-1},\GmodO] \oplus [S^{8k-2},\GmodO]} \arrow[u, bend left,"g'= f' \oplus v^{-1}"]
	\end{tikzcd}
	\]
	Combining it with diagram (\ref{chidisc}) we get
	\[
	\begin{tikzcd}
		& {[M,G]} \arrow[d,"\delta"] \arrow[r,"r'_2"] &	{[S^{4k-1},G]} \arrow[r,"e\oplus D_*"] \arrow[d,"f"] & {[S^{4k-1},G] \oplus [S^{8k-2},G]} \arrow[d,"g"] \\
		R \arrow[r,"\overline{r}"]& {[M,\GmodO]} \arrow[r,"r_2"]&	{[S^{4k-1},\GmodO]} \arrow[r,"e\oplus D_*"] \arrow[u, bend left,"f'"] & {[S^{4k-1},\GmodO] \oplus [S^{8k-2},\GmodO]} \arrow[u, bend left,"g'"]
	\end{tikzcd}
	\]
	where the map $r'_2$ is defined in the same manner as $r_2$ from the exact sequence obtained after applying the contravariant functor $[-,G]$ to the cofibration sequence (\ref{mcof}). Given any $r_2(y) \in [S^{4k-1},\GmodO]$ for some $y \in [M,\GmodO]$, it is mapped to, say, some $\overline{r_2(y)} \in [S^{4k-1},G]$ under $f'$. Since the bottom row is exact, $r_2(y)$ is mapped to zero under the next map $e \oplus D_*$. This implies that $\overline{r_2(y)}$ goes to zero under the top right map $e \oplus D_*$. By exactness of the top row, there exists some $z \in [M,G]$ such that $r'_2(z) = \overline{r_2(y)}$. By commutativity of the first diagram, $r_2(\delta(z))$ and $r_2(y)$ go to same element in $[S^{4k-1},\GmodO]$.
\end{proof}	

Now we shall extensively study the images of $\chi'$ and $\chi''$. In fact from (\ref{mreln1}) it follows that the map $\chi''$ is the map between the cokernls of the maps $r$ and $(-)\cdot e$ and induced from the multiplication by $\theta_k$. Hence $\text{im}(\chi'')$ is the subgroup
\begin{center}
	$\text{im}(\chi'') \cong \theta_k \cdot \left( \Z /e\Z\right)$.
\end{center}
It follows from the commutative diagram (\ref{mreln1}) restricted to cokernels of maps $r$ and $e$ in the second square which becomes:
\[
\begin{tikzcd}
	R \arrow[r,"\overline{r}"] \arrow[d,"\theta_k"] \arrow[dr,"\chi''"] & {[M,\GmodO]} \arrow[d,"\chi"] \\
	\Z/e\Z \arrow[r,"\cong"] & \Z/e\Z
\end{tikzcd}
\]
In fact, $R$ can be written as the direct product $\Z/e\Z \oplus S$, where $S$ is some torsion group, as the map $r$ in diagram (\ref{mreln1}) is a $(3 \times 2)$-matrix with its top left entry $e$. The vertical map is given as multiplication by $\theta_{k}$ on the first component and zero on the rest, since the rest is torsion. All this together implies that $\im(\chi'') \cong \theta_k\cdot(\Z/e\Z)$. Also note that the cardinality of $\im(\chi'')$ as a set is $\overline{e}=e/g$, and being cyclic, it would imply $\im(\chi'') \cong \Z /\overline{e}\Z$. 

Next, we are going to study the image of the map $\chi'$. Note that the map $\chi'$ fits into the following long exact sequence
\begin{center}
	$\cdots \rightarrow [M,\G] \xrightarrow{\chi'} [M,\GmodTOP] \xrightarrow{\Psi} [M,\BTOP] \rightarrow [M,\BG] \rightarrow \cdots$
\end{center}
Therefore $\text{im}(\chi')$ = ker$(\Psi)$.
To describe the kernel of the map $\Psi$, we shall use the subspace $Y = S^{4k-1} \underset{\alpha}{\cup} e^{4k}$ of $M$ defined earlier. It will eventually be proved that there is an injective map from the kernel of $\Psi$ to the kernel of the restriction map $\Psi|_{Y}$.
Consider the following cofibre sequence
\begin{center}
	$\cdots \rightarrow S^{4k-1}\rightarrow Y \rightarrow S^{4k} \xrightarrow{\cdot e} S^{4k} \rightarrow \Sigma Y \rightarrow \cdots $
\end{center}
Similar to diagram (\ref{mreln}) we have a commutative diagram
\[
\begin{tikzcd}
	\cdots \arrow[r] & \left[S^{4k},\GmodTOP\right] \arrow[r] \arrow[d] & \left[S^{4k}, \GmodTOP \right] \arrow[r] \arrow[d] & \left[Y,\GmodTOP\right] \arrow[r] \arrow[d,"\Psi|_{Y}"] & {} \\
	\cdots\arrow[r] & \left[S^{4k},\BTOP\right] \arrow[r] & \left[S^{4k},\BTOP\right] \arrow[r] & \left[Y,\BTOP\right] \arrow[r] & {} \\
	{} & {} \arrow[r] & \left[S^{4k-1},\GmodTOP\right] \arrow[d] \arrow[r]&{\cdots} \\
	{} & {} \arrow[r] & \left[S^{4k-1},\BTOP\right] \arrow[r]&{\cdots}
\end{tikzcd}
\]
By replacing known spaces and maps from the Kervaire-Milnor braid, see part 1. of Lemma~\ref{lem:Brumfiel-result-about-zeta_one-two}, we obtain
\[
\begin{tikzcd}
	\cdots \arrow[r,"0"]& \Z \arrow[r,"\cdot e"] \arrow[d,"{(j_{k}',t_1)}"] & \Z \arrow[r] \arrow[d,"{(j_{k}',t_1)}"]& {[Y,\GmodTOP]} \arrow[r] \arrow[d,"{\Psi|_{Y}}"] & 0 \arrow[d] \\
	\cdots\arrow[r] & \Z \oplus T \arrow[r,"\cdot e"] & \Z \oplus T \arrow[r] & {[Y,\BTOP]} \arrow[r] & {[S^{4k-1},\BTOP]} 
\end{tikzcd}
\]
where $j_{k}'$ is defined earlier in (\ref{equation:nota}). Top exact row implies $[Y,\GmodTOP] \cong \Z/e\Z$. The following commutative diagram will help us studying kernel of $\Psi|_{Y}$. 
\[
\begin{tikzcd}
	0 \arrow[r] & {[Y,\GmodTOP]} \arrow[r,"\cong"] \arrow[d,"\Psi|_{Y}'"] & {[Y,\GmodTOP]} \arrow[r] \arrow[d,"\Psi|_{Y}"] & 0 \arrow[r] \arrow[d] & 0 \\
	0 \arrow[r] & W \arrow[r] & {[Y,\BTOP]} \arrow[r] & W' \arrow[r] & 0  
\end{tikzcd}
\]
where $W \cong \Z /e\Z \oplus T'$, with $T'$ being cokernel of restriction of $e$ onto $T$, is the cokernel of the bottom map of previous diagram and $W'$ is the cokernel of the inclusion of $W$ in $[Y,\BTOP]$.\\

By snake lemma, ker$(\Psi|_{Y}) \cong$ ker$\Psi|_{Y}'$. Let us denote it by $K$. We have 

\begin{center}
	$K = \text{ker}(\pi_{1}\circ \Psi|_{Y}') \cap \text{ker}(\pi_{2}\circ \Psi|_{Y}')$
\end{center}
where $\pi_1$ and $\pi_2$ are projections onto first and second components.\\

And it is clear that ker$(\pi_{1}\circ \Psi|_{Y}') \cong \Z/h'\Z$ where $h' = \text{gcd}(j_{k}',e)$. Therefore ker$(\Psi|_{Y})$ is a subgroup of $\Z/h'\Z$.\\

Finally, it remains to show that there is an injection ker$\Psi \rightarrow$ ker$\Psi|_{Y}$ implying that $\text{im}(\chi') = $ ker$(\Psi)$ is even smaller. The cofibre sequence
\begin{center}
	$Y \rightarrow M \rightarrow S^{8k-1}$	
\end{center} 
induces the commutative diagram where the indicated map in the top row turns out to be an isomorphism: 
\[
\begin{tikzcd}
	0 \arrow[r] \arrow[d] & {[M,\GmodTOP]} \arrow[r,"\cong"] \arrow[d,"\Psi"] & {[Y,\GmodTOP]} \arrow[r] \arrow[d,"\Psi|_{Y}"] & \Z/2\Z \arrow[r,"\rho"] \arrow[d] & {} \\
	\pi_{8k-1}(\BTOP) \arrow[r] & {[M,\BTOP]} \arrow[r,"\Gamma"] & {[Y,\BTOP]} \arrow[r] & \pi_{8k-2}(\BTOP) \arrow[r] & {} 
\end{tikzcd}
\]
To show the statement about the map in the top row, we need to consider two cases. When $n \neq 0$ this map is an injective map from the finite cyclic group $\Z/e\Z$ to itself, and so it is an isomorphism, which immediately gives us another diagram 
\[
\begin{tikzcd}
	0 \arrow[r] &	0 \arrow[r] \arrow[d] & {[M,\GmodTOP]} \arrow[r,"\cong"] \arrow[d,"\Psi"] & {[Y,\GmodTOP]} \arrow[r] \arrow[d,"\Psi|_{Y}"] & 0  \\
	0 \arrow[r] &	\text{ker}(\Gamma) \arrow[r] & {[M,\BTOP]} \arrow[r,"\Gamma"] & \text{im}(\Gamma) \arrow[r] & 0 
\end{tikzcd}
\]

Now the snake lemma implies that there is an exact sequence 
\begin{center}
	$0 \rightarrow \text{ker}(\Psi) \rightarrow \text{ker}(\Psi|_{Y}) \rightarrow \text{ker}(\Gamma) \rightarrow...$
\end{center}

therefore, the map ker$\Psi \rightarrow$ ker$\Psi|_Y$ is injective. When $n=0$, we can apply the same reasoning, but need to provide an additional reason for why the map $[M,G/TOP]\rightarrow [Y,G/TOP]$ is an isomorphism. Note that in this case it is a map $\Z \rightarrow \Z$. The argument is made by showing that the map $\rho \colon [S^{8k-2},\GmodTOP] \cong \Z/2\Z \rightarrow [\Sigma^{-1} M,\GmodTOP]$ to the right in the diagram is injective. This, in turn, is provided by considering a larger diagram where this map fits:
\[
\xymatrix{
	\cdots \ar[r] & [S^{8k-2},\GmodTOP] \ar[r]^{\rho} & [\Sigma^{-1} M,\GmodTOP] \ar[r] & \cdots \\
	0 \ar[r] \ar[u] & [S^{8k-2},\GmodTOP] \ar[r]_-{\cong} \ar[u]^{\cong} & [S^{4k-1} \vee S^{8k-2} ,\GmodTOP] \ar[r] \ar[u] & 0 \ar[u] \\
	& & [S^{4k-1} ,\GmodTOP] \cong 0 \ar[u] 
}
\]

This finishes the proof of the statement that the map ker($\Psi )\rightarrow$ ker($\Psi|_Y$) is injective.

 Also recall that ker$(\Psi) = \text{im}(\chi')$ which is a subgroup, say $A$, of $\Z/h'\Z$ where $h'=\text{gcd}(j_k',e)$. Therefore $\text{im}(\chi) = \langle A, B \rangle$ where $A = \text{im}(\chi') \le \Z/h'\Z$, $B = \text{im}(\chi'') \cong \theta_k\cdot (\Z/e\Z)$. Note that $B$ is a cyclic group of order $\frac{e}{g} = \overline{e}$.

\begin{lemma}
	$A = \text{im}(\chi') \subseteq \text{im}(\chi'') = B$
\end{lemma}

\begin{proof}
	We shall show that $h'|\overline{e}$, which means that $A$, is in fact a subgroup of $B$ and therefore $\text{im}(\chi) = \text{im}(\chi'')$. 
	
	Let us recall the integers involved, their relevant properties, and introduce some more notation. The necessary definitions are in and around \eqref{equation:nota} and \eqref{equation:alter}:
	\begin{itemize}
		\item $g = \gcd (e,\theta_k)$;
		\item $h' = \gcd (e,j'_k)$, which is an odd integer, since $j'_k$ is odd;
		\item $\theta_k = 2^{\alpha} \cdot \theta'_k$ for suitable $\alpha$;
		\item $1 = \gcd (j'_k,\theta'_k)$;
		\item $e = 2^{\beta} \cdot e'$, where $e'$ is odd, for suitable $\beta$;
		\item $\gamma = \min (\alpha,\beta)$.
	\end{itemize}
	
	Then $g = 2^{\gamma} \cdot \gcd (e',\theta'_k)$ and $g' = \gcd (e',\theta'_k)$ is odd, since both $e'$ and $\theta'_k$ are odd. Moreover, for $\delta = \max (0,\beta-\alpha)$, we have
	\[
	\overline{e} = \frac{e}{g} = 2^{\delta} \cdot \frac{e'}{g'}.
	\]
	Note that $h'$ divides $j'_k$ and $g'$ divides $\theta'_k$. Since $1 = \gcd (j'_k,\theta'_k)$ we obtain that $1 = \gcd (h',g')$. Also $h'$ is an odd number that divides $e$ and because it is coprime to $g'$ it also divides $\frac{e'}{g'}$.
\end{proof}

As a corollary we get
\begin{corollary}
	$\text{im}(\chi) = \text{im}(\chi'')$
\end{corollary}
Therefore the cokernel of $\chi$ contains $\frac{e}{e'} = g$ many elements.  This completes the proof of Theorem~\ref{fn}. \qed

Recall the map $\alpha^k_{m,n}$ from~\eqref{equation:alipha}. It can be composed with the isomorphism $\eta^{\TOP} \colon \mathcal{S}^{\TOP}(M^k_{m,n}) \rightarrow  \mathcal{N}^{\TOP}(M^k_{m,n})$. Theorem~\ref{thm:alpha} together with Theorem~\ref{fn} immediately yield:

\begin{theorem}\label{sameimage}
	The images of the maps $\eta^{\TOP} \circ \alpha_{m,n}^k$ and $\chi$ are the same.
\end{theorem}
Finally consider the following commutative diagram
\[
\begin{tikzcd}
	\mathcal{S}^{\DIFF}(M^k_{m,n}) \arrow[r, "\eta^{\DIFF}"] \arrow[d, "F^k_{m,n}"'] & \mathcal{N}^{\DIFF}(M^k_{m,n}) \arrow[d, "\chi"] \\
	\mathcal{S}^{\TOP}(M^k_{m,n}) \arrow[r, "\cong", "\eta^{\TOP}"'] & \mathcal{N}^{\TOP}(M^k_{m,n})
\end{tikzcd}
\]
where $\eta^{\CAT}$ is the map from the surgery exact sequence in $\CAT = \DIFF$ or $\TOP$ category, and the maps $F^k_{m,n}$ and $\chi$ are forgetful maps. We have already studied the map $\chi$, after using the identification by the bottom map $\eta^{\TOP}$, we obtain $\im(\eta^{\TOP} \circ F^k_{m,n}) \subseteq$ $\im(\chi)$. Also notice that $\im(\alpha_{m,n}^k) \subseteq \im(F^k_{m,n})$ since all manifolds are smooth. Hence, Theorem \ref{sameimage} gives us the following:

\begin{corollary}\label{acf}
	We have
	\[
	\im(\eta^{\TOP} \circ \alpha_{m,n}^k) = \im(\chi) = \im(\eta^{\TOP} \circ F^k_{m,n}).
	\]
\end{corollary}

\begin{corollary}\label{cor:image-of-F-is-subgroup}
	The image of $F^k_{m,n}$ is a subgroup.
\end{corollary}

\begin{proof}
	This follows since $\im (\chi)$ is a subgroup, because $\chi$ is a homomorphism.
\end{proof}

\begin{proof}[Proof of Theorem~\ref{thm:the-forgetful-map-on-str-sets}]
	Combine Corollaries \ref{acf} and \ref{cor:image-of-F-is-subgroup} with Theorem \ref{fn}.
\end{proof}

\begin{proof}[Proof of Corollary~\ref{cor:never-surjective}]
	Both $e$ and $\theta_k$ in the statement of Theorem~\ref{thm:the-forgetful-map-on-str-sets} are even, so we always have $g \geq 2$. 
\end{proof}

\begin{proof}[Proof of Corollary~\ref{cor:any-smooth-is-bundle}]
	This follows directly from Corollary~\ref{acf}.
\end{proof}

\begin{proof}[Proof of Corollary~\ref{cor:if-n-theta-coprime-then-half-of-elements-are-bundles}]
	In this case $g=2$.
\end{proof}

\section{Discussion} \label{sec:leftovers}


Some of the technical statements that we proved in the previous section can be used to establish the following interesting result.

\begin{theorem}\label{cor:realization-of-ni-when-n-theta-coprime}
	Let $N$ be a closed $(8k-1)$-dimensional manifold with $k\ge3$ and gcd$(n,\theta_k)=1$. Then there exists a homotopy equivalence $h:N\rightarrow M_{m,n}^k$ with an even normal invariant if and only if $N$ is homeomorphic to $M_{m',n}$ where $m'=m+j_{k}\cdot t$ for some $t \in \Z$. 
\end{theorem}


\Addresses


\begin{thebibliography}{10}
	\bibitem{JA66}
	J. F. Adams, \textit{On the groups $J(X)$}, $IV$, Topology 5:21-71, 1966.
	
	\bibitem{JA73}
	J. F. Adams, \textit{Stable homotopy and generalized homology}, University of Chicago Press, Chicago, 1973.
	
	\bibitem{biswas2026}
	S. Biswas, \textit{Smooth manifolds homotopy equivalent to products of spheres}, arxiv:2606.10239, 2026.
	
	\bibitem{GEB}
	G. E. Bredon, \textit{Topology and Geometry}, Graduate Texts in Mathematics 139, Springer-Verlag Inc., New York, 1993.
	
	\bibitem{GB68}
	G. Brumfiel, \textit{On the homotopy groups of $BPL$ and $PL/O$}, Annals of Mathematics, 88(2):291-311, 1968.
	
	\bibitem{CW2021}
	S. Chang, S. Weinberger, \textit{A course on surgery theory}, Annals of Mathematics Studies 211, Princeton University Press, 2021.
	
	\bibitem{CCPS23}
	A. Conway, D. Crowley, M. Powell, and J. Sixt, \textit{Simply-connected manifolds with large homotopy stable classes}, J. of Australian Math. Soc., 115(2):172-203, 2023.
	
	\bibitem{DC10}
	D. Crowley, \textit{The smooth structure set of $S^p \times S^q$}, Geometriae Dedicata, 148:15-33, 2010.
	
	\bibitem{CE03} 
	D. Crowley, C. M. Escher,\textit{ A Classification of $S^3$-bundles over $S^4$}, Differential Geometry and its Applications, Elsevier, 18(3):363-380, 2003.
	
	
	\bibitem{DL59}
	A. Dold, R. Lashof, \textit{Principal quasi-fibration and fibre homotopy equivalences of the bundles}, Illinois J. Math. 3:285-305, 1959.
	
	
	\bibitem{AH02}
	A. Hatcher, \textit{Algebraic topology}, Cambridge University Press, Cambridge, 2002.
	
	\bibitem{IWX20}
	D.C. Isaksen, G. Wang and Z. Xu, \textit{Stable homotopy groups of spheres}, PNAS, 117(40):24757-24763, 2020.
	
	\bibitem{JW54}
	I. M. James, J.H.C Whitehead, \textit{The homotopy theory of sphere bundle over spheres I}, Proc. London Math. Soc., 3(4):198-218, 1954.
	
	\bibitem{JW55}
	I. M. James, J.H.C Whitehead, \textit{The homotopy theory of sphere bundle over spheres II}, Proc. London Math. Soc., 3(5):148-166, 1955.
	
	
	\bibitem{James(1954)}
	I. M. James, \textit{On the iterated suspension}, Q. J. Math., Oxf. II. Ser.5:1-10, 1954.
	
	
	\bibitem{KS77}
	R. C. Kirby, L. C. Siebenmann, \textit{Foundational essays on topological manifolds, smoothings, and triangulations}, Annals of Mathematics Studies, No. 88, Princeton University Press, 1977.
	
	
	\bibitem{Levine(1983)}
	J. P. Levine: \textit{Lectures on groups of homotopy spheres}, Algebraic and geometric topology (New Brunswick, N.J., 1983):62--95, Springer, 1985.
	
	\bibitem{LM24}
	W. L\"uck, T. Macko, \textit{Surgery Theory: Foundation}, Grundlehren der
	mathematischen Wissenschaften, Volume 362, Springer Switzerland, 2024
	
	\bibitem{MM}
	I.B. Madsen, R.J. Milgram, \textit{The Classifying Spaces For Surgery and Cobordism of Manifolds}, Princeton University Press, New Jersey, 1979.
	
	\bibitem{Mahowald(1965)} M. Mahowald, \textit{Some {Whitehead} products in {{\(S^n\)}}}, Topology, 4:17-26, 1965. 
	
%
	
	\bibitem{RM24}
	A. Raj, T. Macko, \textit{On Manifolds Homotopy Equivalent to the Total Spaces of $S^7$-Bundles Over $S^8$}, Archivum Mathematicum, 60(3):125-134, 2024.
	
%
%
	
%
%
%
	
	\bibitem{Wall99}
	C.T.C. Wall, \textit{Surgery on Compact Manifolds}, Second Edition, Mathematical Surveys and Monographs, 69, AMS, 1999.
	
	\bibitem{GWW}
	G.W. Whitehead, \textit{On Products in Homotopy Groups}, Annals of Mathematics, 40(3):460-475, July 1946.
	
	
	\bibitem{Whitehead(1978)}
	G. W. Whitehead, \textit{Elements of homotopy theory}, GTM {61}, {Springer Cham} {1978}.
	
	\bibitem{zhu-pan(2026)}
	Zhongjian Zhu and Jianzhong Pan, \textit{Homotopy classification of $S^{2k-1}$-bundles over $S^{2k}$}, arXiv:2508.14341, 2026.
\end{thebibliography}
\end{document}